\documentclass[12pt]{amsart} 
\usepackage{amsmath,amsthm}
\usepackage{amssymb}
\usepackage{enumerate,url}

\newcommand{\ph}{\varphi}
\newcommand{\eps}{\varepsilon}
\renewcommand{\gcd}{\operatorname{gcd}}

\newcommand{\zeile}{\vspace{\baselineskip}}

\newcommand{\N}{\mathbb{N}}

\newcommand{\Z}{\mathbb{Z}}
\newcommand{\R}{\mathbb{R}}
\newcommand{\C}{\mathbb{C}}

\newcommand{\e}{\mathrm{e}}

\newcommand{\uq}{\mathbf{q}}
\newcommand{\ux}{\mathbf{x}}

\newcommand{\md}{\textup{~mod~}}

\newtheorem{theorem}{Theorem}[section]
\newtheorem{lemma}[theorem]{Lemma}

\newtheorem{assumptions}[theorem]{Assumptions}
\newtheorem{conjecture}[theorem]{Conjecture}

\allowdisplaybreaks

\begin{document}

\title[A Bombieri--Vinogradov Theorem with Gaussian primes]{A Bombieri--Vinogradov Theorem with products of Gaussian primes as moduli}

\author[K.~Halupczok]{Karin Halupczok}

\address{Karin Halupczok,  Mathematisches Institut, 
WWU M\"{u}nster,
Einsteinstra\ss{}e 62, D-48149 M\"{u}nster, Germany}

\email{karin.halupczok@uni-muenster.de}

\date{}

\begin{abstract}
We prove a version of the Bombieri--Vinogradov Theorem
with certain products of Gaussian primes as moduli,
making use of their special form as 
polynomial expressions in several variables.
Adapting Vaughan's proof of the classical 
Bombieri--Vinogadov Theorem, cp.~\cite{VauT} to this setting,
we apply the polynomial large sieve inequality that
has been proved in \cite{KH} and which includes
recent progress in Vinogradov's mean value theorem
due to Parsell \emph{et al.}\ in \cite{PPW}.
From the benefit of these improvements,
we obtain an extended range for the variables
compared to the range obtained from
standard arguments only. 
\end{abstract}


\maketitle

\section{Introduction}
\label{sec:intro}

The classical theorem of Bombieri--Vinogradov states that
\begin{theorem}[Theorem of E.~Bombieri \cite{Bo} and 
A.~I.~Vinogradov \cite{AIV1,AIV2}]
\label{thm:BVclass}
For any $A,Q,x> 1$ we have
\[
   \sum_{q\leq Q} \sup_{y\leq x} \max_{a \md q} 
   |E(y;q,a)|\ll_{A} \frac{x}{(\log x)^{A}} + 
   Q\sqrt{x}(\log (Qx))^{3}
\]
where 
\[
E(y;q,a):=\psi(y;q,a)-\frac{y}{\ph(q)}
\text{ and }
\psi(y;q,a):=\sum_{\substack{n\leq y\\ n\equiv a\md q}} \Lambda(n),
\]
\end{theorem}
so the nontrivial upper bound $\ll_{A} x(\log x)^{-A}$ is
obtained for $Q\leq x^{1/2}(\log x)^{-3-A}$.
It is well known to be a difficult task to break the ``$1/2$-barrier'',
which means to show the estimate for $q$ bigger that
$x^{1/2}$.
The famous Elliott--Halberstam conjecture in \cite{EHconj} states that 
the estimate
should hold even with $Q\ll x^{1-\eps}$. Many applications, especially
recent progress in the solution of the small gap conjecture,
rely on such improvements on the bound $Q$ for certain moduli sets 
for $q$.
It has been found by Y.~Zhang in \cite{Zha} that a restriction to certain
smooth moduli breaks the $1/2$-barrier.

A standard approach to prove Bombieri--Vinogradov's theorem
is the proof of Vaughan, cp.~ \cite{VauT}, by making use of the large sieve
inequality. In \cite{KH}, nontrivial improvements of the large
sieve inequality with polynomial moduli have been achieved, based
on the work of Parsell \emph{et al.}\ in \cite{PPW} 
in connection with progress in Vinogradov's mean value theorem. 
The results of \cite{KH} are superior to standard approaches 
in a number of applications where the degree
of the considered polynomial is bigger than the number of variables.

In this article, we use Vaughan's approach to prove a variant
of the Bombieri--Vinogradov Theorem 
with polynomial moduli having certain properties.
The polynomial behavior of the moduli is 
exploited in the proof by using the result in \cite{KH}.
Our main result is the following:

\begin{theorem}[A Bombieri--Vinogradov Theorem with special 
polynomial moduli]
\label{th:BV}  
Let $A,Q,x> 1$ be real and $\ell,k\geq 1$ be integers.
Consider two maps
 $u,v:\{1,\dots,k\}\to\{1,\dots,\ell\}$ such that
$\{u(i),v(i)\}\neq\{u(j),v(j)\}$
for $i\neq j$,
and let $P\in\Z[x_{1},\dots,x_{\ell}]$ be the
polynomial $P(\ux):=\prod_{i=1}^{k}(x_{u(i)}^{2}+x_{v(i)}^{2})$.

Let $\sigma=1/(4kr)$ with $r:=\binom{2k+\ell-1}{\ell}-1$ 
and suppose that
\begin{equation}
\label{eq:QIV}
x^{\eps/\sigma}\ll Q\leq
x^{(1/3-2\eps)/(2k-\sigma)}
\end{equation}
for an arbitrary $\eps>0$.
Then we have the estimate
\[
\sum_{\uq\sim Q} G_{\uq}\frac{\ph(P(\uq))}{Q^{\ell}}
   \sup_{y\leq x} \max_{\substack{a\md P(\uq)\\\gcd(a,P(\uq))=1}}
    |E(y;P(\uq),a)| \ll_{A,\ell,k,\eps} \frac{x}{(\log x)^{A}},
\]
where 
\[
   G_{\uq}:= \mu^{2}(P(\uq))
   \Lambda(q_{u(1)}^{2}+q_{v(1)}^{2})\cdots
   \Lambda(q_{u(k)}^{2}+q_{v(k)}^{2}),
\]
and the sum runs
over all $\uq$ with $Q<q_{i}\leq 2Q$, $i=1,\dots,k$.
\end{theorem}

In the weights $G_{\uq}$, the $\Lambda$-arguments are all
primes $\equiv 1\md 4$ by Gauss' theorem on the sum of
two squares, 
and $P(\uq)$ is a squarefree number composed of such Gaussian primes.
Hence we consider certain subsets of
\begin{multline*}
  \{p_{1}\cdots p_{k};\text{ all } p_{i}\equiv 1\md 4 
 \text{ prime,}\\ \text{pairwise different and } 2Q^{2}<p_{i}\leq 8Q^{2} \}
\end{multline*}
as moduli, which is the set of squarefree products of $k$ Gaussian 
primes from a certain interval of length $6Q^{2}$.
Note that the subsets we consider
become quite sparse if the degree $2k$ of
$P$ is bigger than the number $\ell$ of variables.
In that case, the estimate in Theorem~\ref{th:BV} can
not be deduced directly by applying 
the classical Theorem~\ref{thm:BVclass}
due to the assigned
weights that reflect the sparsity of the moduli. 
In other words, Theorem~\ref{th:BV} takes the distribution of the
moduli into account,
whereas the classical theorem only sees the size of the moduli.

It is clear from the proof that the range with exponent $1/6k$ 
in \eqref{eq:QIV}
could also be obtained using a standard
approach. The improvement here is the term $3\sigma$ coming from the
benefit of the polynomial large sieve inequality 
from \cite{KH}. So to speak, it breaks the ``$1/6k$-barrier''.
Whether this ``$1/6k$-barrier'' is really such a hard barrier 
is not that clear since by the classical theorem, 
one might heuristically expect a hard barrier at $1/4k$.

A similar phenomenon is already known from the
literature in  the case of polynomials 
in one variable of degree $d$; in that case, heuristically,
$1/2d$ might be reached. This has been
investigated by  Elliott in \cite{Ell}, 
who proved such a Bombieri--Vinogradov-type
theorem with exponent $1/4d$ and gave evidence that 
one might be able to reach $1/3d$ by further improvements,
though the barrier of the method seems to
be at $1/4d$. Later,
Mikawa and Peneva \cite{MikPen} improved the exponent to $8/19d$, and
Baker \cite{Bak} to $9/20d$.

For the polynomials of degree $d=2k$ considered in this article,
Theorem~\ref{th:BV} confirms the exponent $1/3d$,
even improving it in a way depending on $\sigma$.
Note however, that we always have at least two variables.
It is not clear yet whether the improvements in \cite{MikPen, Bak}
can be extended to a several variable setting as in the present article.
The proof of Theorem~\ref{th:BV} relies on the classical 
Fourier Analysis approach, whereas deeper such techniques are used 
in \cite{MikPen} and \cite{Bak}.
Elliott's argument in \cite{Ell} is a largely self-contained 
careful application of Linnik's
Dispersion Method, without appeal to Fourier Analysis.

We continue by giving some important comments on our choice
of the moduli.
 
Firstly, the proof of the polynomial version of the Basic Mean
Value Theorem in Section~\ref{sec:polyBMVT} does \emph{not} 
depend on this choice: it works for arbitrary polynomials
$P$ of degree $k$ in $\ell$ variables, assuming only 
that the biggest value $M_{Q}$ and smallest value $m_{Q}$ of $P$
in the dyadic $Q$-box $\uq\sim Q$ are such that 
$Q^{k}\ll m_{Q}\leq M_{Q}\ll Q^{k}$ holds for $P$.
In the proof, one needs primes $p=q_{u}^{2}+q_{v}^{2}$
with $q_{u},q_{v}$ of similar size.
But then, Theorem~\ref{th:BV} and its proof 
rely on the special structure of the moduli:
we use that each divisor is again 
of such a form so that the polynomial basic mean value theorem
can be used iteratively.

Secondly, one should make clear that the estimate in our 
Theorem~\ref{th:BV} is nontrivial in the sense that
the number of moduli is big enough and not too sparse,
so that it can not be deduced using the trivial estimate
$E(y;P(\uq),a)\ll y/\ph(\uq)$.

This would indeed follow from the following conjecture.

\begin{conjecture}[Number of moduli]
\label{conj:num}
Consider $u,v,P$ as in Theorem~\ref{th:BV}.
There exists a constant $C>0$ depending on $k$ and $\ell$ only such that
\[
   \sum_{\uq\sim Q} \mu^{2}(P(\uq))
   \Lambda(q_{u(1)}^{2}+q_{v(1)}^{2})\cdots
   \Lambda(q_{u(k)}^{2}+q_{v(k)}^{2}) \gg_{k,\ell} \frac{Q^{\ell}}{(\log Q)^{C}}.
\]
\end{conjecture}

As a second result of this article, we confirm this conjecture 
in certain cases, namely
when not too many of the Gaussian primes 
share a summand $q_{i}^{2}$.

\begin{theorem}[Special cases]
  \label{th:cases}  
Assume that the maps $u,v:\{1,\dots,k\}\to\{1,\dots,\ell\}$ are such
that for each $i\leq k$, one of the numbers $u(i)$ and $v(i)$ 
does not occur in the set $\{u(i+1),\dots,u(k), v(i+1),\dots,v(k)\}$.
Then the assertion in Conjecture~\ref{conj:num} holds true.
\end{theorem}
To give an example, if the sequence of the pairs $(u(i),v(i))$
is $(1,2)$, $(2,3)$, $(3,4)$, we deal with the polynomial
$P(\ux)=(x_{1}^{2}+x_{2}^{2})(x_{2}^{2}+x_{3}^{2})(x_{3}^{2}+x_{4}^{2})$.
On the other hand, the sequence of pairs $(1,2),(2,3),(3,1)$ 
does not comply with the condition.
Clearly, if this condition holds, 
then $\ell\geq k+1$. However,
the degree $2k$ of the polynomial may still be bigger than the
number $\ell$ of variables, so that Theorem~\ref{th:BV}
is nontrivial. 

We prove Theorem~\ref{th:cases} in Section~\ref{sec:last}
by making use of the main theorem of Fouvry and Iwaniec in \cite{FI}.

Some other additional remarks on Theorem~\ref{th:BV}:
\begin{enumerate}[(i)]
\item 
Further improvements of Theorem~\ref{th:BV} 
could be made if the relevant
term $Q^{\ell-\sigma}N$ in the polynomial large sieve
inequality, which is dominant in the relevant ranges,
could be further improved. Ideas how this could be
reached, but showing also its difficulty, are discussed in
\cite[Sec.~5]{KH}.
\item
The restriction $\uq\sim Q$ can be
generalized to $R\ll \uq\ll Q$. In that case the estimate in 
Theorem~\ref{th:BV} holds with upper bound 
$Q\ll x^{1/6k-\eps}R^{\sigma/2k}$. 
Therefore the benefit coming from $\sigma$ melts if $R$
decreases. The theorem gives the biggest possible upper bound 
for $Q$ if $Q/R\ll 1$.
\item
It should be possible to obtain a power of $\log x$ instead of the
term $x^{\eps}$ in the range for $Q$ by working more precisely.
This would require a refinement of the used theorem \cite[Thm.~10.1]{PPW} of 
Parsel \emph{et al.}\ where $Q^{\eps}$ is replaced 
by a power of $\log Q$. 
\end{enumerate}

\zeile
\textbf{Notation.} Let $k,\ell$ denote positive integers and
let $\eps$ denote a 
positive real number.
In this article,
we suppress the dependence of the implicit constants
on $k$, $\ell$ or $\eps$ in our notation and simply write 
$\ll$ for $\ll_{k,\ell,\eps}$ or $\ll_{k,\ell}$.

For a real number $Q> 1$
the symbol $q\sim Q$ means $Q< q\leq 2Q$,
and the notation $\uq\sim Q$ means that the $\ell$-tuple 
$\uq=(q_{1},\dots,q_{\ell})$ of integers 
is contained in a dyadic $Q$-box, that is $q_{i}\sim Q$ for
$i=1,\dots,\ell$. 

For $\alpha\in\R$, the symbol $\e(\alpha):=\exp(2\pi i \alpha)$ denotes 
the complex exponential function.
The greatest common divisor is abbreviated by $\gcd$.
As usual, we denote 
the von Mangoldt function by $\Lambda$, Euler's totient function
by $\ph$, and the M\"obius function by $\mu$.

\section{Auxiliary tools}
\label{sec:polyLSI}

\begin{assumptions}\label{Assu}
Let $\ell$ be a positive integer and let $P\in\Z[x_{1},\dots,x_{\ell}]$
be a polynomial in $\ell$ variables of degree $k\geq 2$.
Let $\sigma:=1/(2rk)$ with $r:=\binom{k+\ell-1}{\ell}-1$,
and for a real number $Q> 1$
consider the $\ell$-tuples 
in the dyadic $Q$-box $\uq\sim Q$. 

Assume that $P$ takes only positive values in the $Q$-box
and that the biggest value $M_{Q}$ and smallest value $m_{Q}$ of $P$
in this box are such that 
$Q^{k}\ll m_{Q}\leq M_{Q}\ll Q^{k}$ holds for $P$.
\end{assumptions}

In \cite[Cor.~3]{KH}, we obtained the following
polynomial large sieve inequality.

\begin{theorem}[Polynomial large sieve inequality]
\label{polyLSI}
  Let $P$ be a polynomial as in Assumptions \ref{Assu},
  let $(v_{n})$ be a complex sequence, and let 
 \[S(\alpha):=\sum_{M<n\leq M+N} v_{n}\e(\alpha)\] and
  \[
     \Sigma=\Sigma_{Q,N,P}:= \sum_{\uq\sim Q} \sum_{\substack{1\leq
         a\leq P(\uq)\\\gcd(a,P(\uq))=1}}
      \bigg| S\Big(\frac{a}{P(\uq)}\Big)\bigg|^{2}.
  \]
  Then we have the bound
  \begin{equation}
    \label{eq:genpolyLSI}
    \Sigma \ll 
(QN)^{\eps} \cdot \Delta(Q,N) \cdot\sum_{M<n\leq M+N} |v_{n}|^{2}
 \end{equation}
with $\Delta(Q,N):=
Q^{k+\ell}+Q^{\ell-\sigma}N+Q^{\ell+k\sigma}N^{1-\sigma}$.
\end{theorem}

Further, the following Lemmas are standard tools in the proof
of the classical Bombieri--Vinogradov Theorem: their proofs
can be found in the literature, see e.\,g.\ \cite{VauT}.

\begin{lemma}[Consequence of Vaughan's identity]
\label{Vauid}
  Let $U,x\geq 1$, $U^{2}\leq x$, $f:\N\to\C$. Then
\[
   \sum_{U<n\leq x}f(n)\Lambda(n) \ll (\log x)T_{1}+T_{2}+T_{3}
\]
with
\[
   T_{1}=\sum_{\ell\leq U} \max_{w} 
   \Big|\sum_{w<k\leq x/\ell} f(k\ell) \Big|
\]
and
\[
  T_{i}=\Big|\sum_{U<m\leq\max\{x/U,U^{2}\}} a_{i}(m) b_{i}(k) f(mk)
\Big| \text{ for }i=2,3,
\]
where $a_{i}(m)$, $b_{i}(k)$ are arithmetic functions depending
on $U$ only and $|b_{i}(k)|\leq\sum_{d\mid k}1$, $|a_{i}(k)|\leq \log
k$ for all $k\in\N$.
\end{lemma}
Lemma~\ref{Vauid} is presented as 
\cite[Satz~6.1.1]{Brue} in the book of Br\"udern. It can be deduced 
easily from the widely-known Vaughan identity.

\zeile
\begin{lemma}[Polya--Vinogradov's inequality]
\label{pvi} Let $w<z$ be real.
  For any nonprincipal character $\chi$ mod $q>1$ we have
\[
   \sum_{w<k\leq z} \chi(k) \ll q^{1/2}\log q.
\]
\end{lemma}

\begin{lemma}[Formula for $\chi(n)$]
\label{chiformula}
  Let $q\geq 1$. Then for all $n\in\Z$ and all primitive characters
$\chi$ mod $q$ we have
\[
\chi(n)\tau(\overline{\chi}) = \sum_{h\md q} 
\overline{\chi}(h)\e(hn/q), 
\]
where $\tau(\chi):=\sum_{a=1}^{q}\chi(a)e(a/q)$ is the
Gaussian sum. We have $|\tau(\chi)|=\sqrt{q}$ for primitive $\chi$ mod
$q$.
\end{lemma}

\begin{lemma}[Get rid of $mn\leq X$]
\label{ridofrestr}
  Let $T,X> 1$ be real, $M,N\geq 1$ be integers and 
let $(\gamma_{n})$, $(\eta_{n})$ be complex
sequences. Then
\begin{multline}
    \sum_{\substack{m\leq M, n\leq N \\ mn\leq X}} \gamma_{m} \eta_{n} 
   \ll \int_{-T}^{T} \Big| \sum_{m\leq M} \gamma_{m} m^{it}
   \sum_{n\leq N} \eta_{n} n^{it} \Big| \min(|t|^{-1}, \log (2MN)) dt \\
   + MNT^{-1}\sum_{m\leq M, n\leq N} |\gamma_{m} \eta_{n}|.
 \end{multline}
\end{lemma}

\section{The polynomial large sieve 
inequality for characters and its bilinear version} 

\begin{lemma}[Polynomial large sieve inequality with characters]
\label{multLSI}
Let $Q,x>1$ and $(v_{n})$ be a complex sequence. Then
\[
   \sum_{\uq\sim Q} \frac{P(\uq)}{\ph(P(\uq))} 
   \sideset{}{^{*}}\sum_{\chi(P(\uq))} \Big| 
   \sum_{n\leq x}v_{n}\chi(n)\Big|^{2}
   \ll(Qx)^{\eps}\cdot \Delta(Q,x)\cdot \sum_{n\leq x} |v_{n}|^{2},
\]
where the star means that the sum is stretched over all
primitive characters $\chi$ mod $P(\uq)$.
\end{lemma}

\begin{proof}
  If $\chi$ mod $P(\uq)$ is primitive, Lemma~\ref{chiformula} gives
\[
  \Big|\sum_{n\leq x} v_{n} \chi(n)\Big|^{2}
 = \frac{1}{P(\uq)} \Big|\sum_{a=1}^{P(\uq)} \sum_{n\leq x}
 \overline{\chi}(a) \e\Big(\frac{an}{P(\uq)}\Big)v_{n} \Big|^{2}.
\]
We sum this equation on the right hand side over all characters,
and on the left hand side over all primitive characters. We obtain
\begin{align*}
 \sideset{}{^{*}}\sum_{\chi(P(\uq))}
   &\Big|\sum_{n\leq x} v_{n} \chi(n)\Big|^{2}\\ &\leq
   \frac{1}{P(\uq)} \sum_{\chi(P(\uq))} \Big|\sum_{a=1}^{P(\uq)} \sum_{n\leq x}
 \overline{\chi}(a) \e\Big(\frac{an}{P(\uq)}\Big)v_{n} \Big|^{2} \\
  &= \frac{1}{P(\uq)}\sum_{a,c=1}^{P(\uq)} \sum_{m,n\leq x} \sum_{\chi(P(\uq))} 
  \overline{\chi}(a) \chi(c)
  \e\Big(\frac{an-cm}{P(\uq)}\Big)v_{n}\overline{v_{m}} \\
  &= \frac{\ph(P(\uq))}{P(\uq)}\sum_{\substack{a\md P(\uq)\\
      \gcd(a,P(\uq))=1}} \sum_{m,n\leq x} 
  \e\Big(\frac{a(n-m)}{P(\uq)}\Big)v_{n}\overline{v_{m}} \\
  &= \frac{\ph(P(\uq))}{P(\uq)}\sum_{\substack{a \md P(\uq)\\ \gcd(a,P(\uq))=1}} 
   \Big|S\Big(\frac{a}{P(\uq)}\Big)\Big|^{2},
\end{align*}
hence 
\[
   \sum_{\uq\sim Q}\frac{P(\uq)}{\ph(P(\uq))}
\sideset{}{^{*}}\sum_{\chi (P(\uq))}  
\Big|\sum_{n\leq x} v_{n} \chi(n)\Big|^{2}\ll   
   (Qx)^{\eps}\cdot\Delta(Q,x)\cdot\sum_{n\leq x}|v_{n}|^{2} 
\] by Theorem~\ref{polyLSI}.
\end{proof}

\begin{lemma}[Bilinear inequality]
\label{le:bilin}
   Let $x,Q,M,N> 1$, let $(a_{m})$ and 
   $(b_{n})$ be complex sequences. 
   Then
\begin{multline*}
  \sum_{\uq\sim Q}
   \frac{P(\uq)}{\ph(P(\uq))}\sideset{}{^{*}}\sum_{\chi(P(\uq))}
   \max_{X} \Big| \sum_{\substack{m\leq M\\n\leq N\\mn\leq X}} a_{m}b_{n}\chi(mn)\Big| \\
  \ll (QMN)^{\eps}\cdot \Big(\Delta(Q,M)
   \Delta(Q,N)\cdot \sum_{m\leq M}|a_{m}|^{2} 
  \sum_{n\leq N} |b_{n}|^{2}\Big)^{1/2},
 \end{multline*}
where the star means that the sum runs over all primitive characters.
\end{lemma}

\begin{proof}
  Let $A(t,\chi):=\sum_{m\leq M} a_{m}\chi(m)m^{it}$,
$B(t,\chi):=\sum_{n\leq N} b_{n} \chi(n) n^{it}$
and write $\|a\|:=(\sum_{m\leq M}|a_{m}|^{2})^{1/2}$ and
$\|b\|:=(\sum_{n\leq N}|b_{n}|^{2})^{1/2}$.
Then using Lemma~\ref{ridofrestr} with $\gamma_{m}=a_{m}\chi(m)$,
$\eta_{n}=b_{n}\chi(n)$ and summing up, we bound
the left hand side of the lemma by
\begin{align*}
  &\ll \int_{-T}^{T} \sum_{\uq\sim Q} \frac{P(\uq)}{\ph(P(\uq))}
  \sideset{}{^{*}}\sum_{\chi(P(\uq))} |A(t,\chi)B(t,\chi)|
  \min(|t|^{-1},\log (2MN) )dt \\&\hspace{4cm}+ 
  \sum_{q\leq Q}P(\uq)MNT^{-1}\sum_{m\leq
  M, n\leq N} |a_{m} b_{n}| \\
  &\ll (QMN)^{\eps}\Delta(Q,M)^{1/2} \Delta(Q,N)^{1/2} \|a\|\|b\|
  \int_{-T}^{T}\Xi(t)dt \\
  &\hspace{4cm}+ Q^{k+\ell}(MN)^{3/2}T^{-1}
  \|a\| \|b\|,
\end{align*}
where Lemma~\ref{multLSI} and the Cauchy--Schwarz inequality
has been used in the second estimate.
The assertion follows with $T=(MN)^{3/2}$.
\end{proof}

\section{Proof of the polynomial Basic Mean Value Theorem}
\label{sec:polyBMVT}
In this section, we prove the polynomial version of the Basic Mean
Value Theorem using the polynomial large sieve.

\begin{theorem}[Polynomial Basic Mean Value Theorem]
\label{thm:polyBMVT}
  Let $Q,x>1$, let $P$ be a polynomial and $\sigma>0$ 
  as in Assumptions~\ref{Assu}, and for a primitive character $\chi$ mod $q$ 
  we write $\psi(x,\chi):=\sum_{n\leq x} \chi(n) \Lambda(n)$. Then
\begin{equation}
\label{eq:polyBMVT}  
   \sum_{\uq\sim Q} \frac{P(\uq)}{\ph(P(\uq))}
     \sideset{}{^{*}}\sum_{\chi(P(\uq))} \sup_{y\leq x}|\psi(y,\chi)|
   \ll (Qx)^{\eps} \cdot\tilde{\Delta}(Q,x) 
\end{equation}
with
\[
   \tilde{\Delta}(Q,x) :=
     Q^{\ell-\sigma}x+Q^{\ell+(k-\sigma)/2}x^{5/6}+Q^{\ell+(k-1)\sigma/2}x^{1-\sigma/6},
     \text{if }x\geq Q^{2k+\sigma}, 
\]
and
\[
\tilde{\Delta}(Q,x) :=
     Q^{\ell+5k/6-\sigma/3}x^{2/3}, \text{if }Q^{k+3-\sigma}\leq x\leq Q^{2k+\sigma}.
\]
\end{theorem}

The second range for $x\leq Q^{2k+\sigma}$ is not relevant in the
proof of Theorem~\ref{th:BV}, but the result and proof in this case is
included here for completeness.

\begin{proof}
  Let $y=y(\chi)\leq x$ be such that $|\psi(y,\chi)|=\max_{z\leq
    x}|\psi(z,\chi)|$. Let $U\geq 1$, $U^{2}\leq x$. Then Vaughan's
  identity in the form of Lemma~\ref{Vauid} yields
\[
   |\psi(y,\chi)|\ll U+ (\log x) T_{1}(\chi) +T_{2}(\chi) +T_{3}(\chi),
\]  
where
\[
   T_{1}(\chi)=\sum_{r\leq U} \max_{w} \Big|\sum_{w<s\leq y/r} \chi(sr)\Big|,
\]
with
\[
   T_{i}(\chi)=\Big|\sum_{\substack{m>U,\\m\leq\max(U^{2},x/U)}}
   \sum_{s\leq x/m} a_{i}(m)b_{i}(s)\chi(sm) \Big|,\ i=2,3.
\]
Choosing $U$ such that it depends on $Q$ and $x$ only, we can sum over
all primitive $\chi$ and $\uq$. We obtain
\[
  \sum_{\uq\sim Q} \frac{P(\uq)}{\ph(P(\uq))} 
  \sideset{}{^{*}}\sum_{\chi(P(\uq))} |\psi(y(\chi),\chi)| \ll
  (UQ^{\ell+k} +K_{1}\log x +K_{2} +K_{3})\log x,
\]
where 
\[
   K_{j}:=\sum_{\uq\sim Q} \frac{P(\uq)}{\ph(P(\uq))} 
  \sideset{}{^{*}}\sum_{\chi(P(\uq))} T_{j}(\chi),\ j=1,2,3.
\]
We estimate $K_{1}$ using Polya--Vinogradov's inequality
Lemma~\ref{pvi} as
\[
   K_{1}\ll U\sum_{\uq\sim Q} P(\uq)^{3/2} \log^{2} x \ll Q^{\ell+3k/2}
   U\log^{2} x,
\]
which already dominates the term $UQ^{\ell+k}\log x$ above.

We proceed to estimate $K_{2}+K_{3}$. Let $M\leq x$. We will
use dyadic summation over $M$. For arithmetic functions $a$, $b$
we consider the expression 
\[
   K_{M}:= \sum_{\uq\sim Q} \frac{P(\uq)}{\ph(P(\uq))} 
  \sideset{}{^{*}}\sum_{\chi(P(\uq))} \Big| \sum_{m\sim M}
   \sum_{s\leq x/m} a(m)b(s)\chi(sm)\Big|.
\]
Writing the conditions of summation over $s$ as $s\leq x/M$,
$ms\leq x$, we apply the bilinear inequality of Lemma ~\ref{le:bilin}
(choosing $a(m)=0$ for $m\leq M$) which yields
\[
  K_{M}\ll (Qx)^{\eps} \Big(\Delta(Q,M) \Delta(Q,x/M)
  \sum_{m\sim M}|a(m)|^{2}
  \sum_{s\leq x/M}|b(s)|^{2}\Big)^{1/2}.
\]
Now use $\sum_{m\leq 2M}|a(m)|^{2}\ll M(\log M)^{2}$,
$\sum_{s\leq z}|b(s)|^{2}\ll z(\log z)^{3}$: this yields
\[
   K_{M}\ll x^{1/2+\eps}\Delta(Q,M)^{1/2}\Delta(Q,x/M)^{1/2}
\text{ for }M\leq x,
\]
hence
\begin{align*}
 K_{M}&\ll x^{1/2+\eps} (Q^{\ell+k}
 +Q^{\ell-\sigma}M+Q^{\ell+k\sigma}M^{1-\sigma})^{1/2} \\
 &\hspace{1cm} \cdot(Q^{\ell+k}
+Q^{\ell-\sigma}xM^{-1}+Q^{\ell+k\sigma}(x/M)^{1-\sigma})^{1/2}\\
&\ll x^{1/2+\eps} (Q^{2\ell+2k}+Q^{2\ell+k-\sigma}xM^{-1}+
 Q^{2\ell+k+k\sigma}(x/M)^{1-\sigma} \\ &\hspace{1cm}+ Q^{2\ell+k-\sigma}M
+ Q^{2\ell-2\sigma}x+Q^{2\ell+(k-1)\sigma}x^{1-\sigma}M^{\sigma}\\
&\hspace{1cm}+Q^{2\ell+k+k\sigma}M^{1-\sigma}+Q^{2\ell+(k-1)\sigma}xM^{-\sigma}
 +Q^{2\ell+2k\sigma}x^{1-\sigma})^{1/2}.
\end{align*}
Now the dyadic summation for $M=2^{\nu}U$, $M\leq W\leq x$,
with $\nu=0,1,2,\dots$ yields
\begin{align*}
  \sum_{\uq\sim Q} &\frac{P(\uq)}{\ph(P(\uq))} 
  \sideset{}{^{*}}\sum_{\chi(P(\uq))}
  \Big|\sum_{U<m\leq W} \sum_{s\leq x/m} a(m)b(s)\chi(sm) \Big| \\
   &\ll x^{1/2+\eps}(Q^{2\ell+2k}+Q^{2\ell+k-\sigma}xU^{-1}+
 Q^{2\ell+k+k\sigma}(x/U)^{1-\sigma} \\ &\hspace{1cm}+ Q^{2\ell+k-\sigma}W
+ Q^{2\ell-2\sigma}x+Q^{2\ell+(k-1)\sigma}x^{1-\sigma}W^{\sigma}\\
&\hspace{1cm}+Q^{2\ell+k+k\sigma}W^{1-\sigma}+Q^{2\ell+(k-1)\sigma}xU^{-\sigma}
 +Q^{2\ell+2k\sigma}x^{1-\sigma})^{1/2}.
\end{align*}
Choosing $W=\max(U^{2},x/U)$, $a=a_{i}$, $b=b_{i}$ for $i=2,3$, we
bound $K_{2}$ and $K_{3}$ by
\begin{align*}
  K_{2}+K_{3}&\ll x^{1/2+\eps}(Q^{2\ell+2k}
+ Q^{2\ell-2\sigma}x+Q^{2\ell+2k\sigma}x^{1-\sigma}\\
&\hspace{1cm}+Q^{2\ell+k-\sigma}xU^{-1}+
 Q^{2\ell+k+k\sigma}(x/U)^{1-\sigma} +Q^{2\ell+(k-1)\sigma}xU^{-\sigma}\\ 
&\hspace{1cm}+ 
Q^{2\ell+k-\sigma}U^{2}+Q^{2\ell+k+k\sigma}U^{2(1-\sigma)}
+Q^{2\ell+(k-1)\sigma}x^{1-\sigma}U^{2\sigma})^{1/2}.
\end{align*}

Together with $K_{1}\ll x^{\eps}Q^{\ell+3k/2}U$, we optimize the terms
depending on the two ranges $Q\ll x^{1/(2k+\sigma)}$ and
$x^{1/2k+\sigma}\ll Q\ll x^{1/(k+3-\sigma)}$ by choosing $U$ suitably
to obtain the bounds stated in the theorem.

\textbf{First range:} $Q^{2k+\sigma}\leq x$. In that case, we choose
$U=x^{1/3}$ so that $U^{2}=x/U$, this yields
\begin{align*}
  K_{2}+K_{3}&\ll x^{\eps}(Q^{\ell+k}x^{1/2} + 
 Q^{\ell-\sigma}x+Q^{\ell+k\sigma}x^{1-\sigma/2}\\
&+Q^{\ell+(k-\sigma)/2}x^{5/6}+Q^{\ell+k/2+k\sigma/2}x^{5/6-\sigma/3}
+ Q^{\ell+(k-1)\sigma/2}x^{1-\sigma/6}),
\end{align*}
and this also bounds $K_{1}$ since 
$K_{1}\ll x^{\eps}Q^{\ell+3k/2}x^{1/3}\ll
Q^{\ell+k/2-\sigma/2}x^{5/6}$ holds in the assumed first range.
Further, in this bound for $K_{2}+K_{3}$, 
we can leave out the first, third and
fifth summand since a simple calculation shows that
they are dominated by the fourth and sixth.

So, in the assumed range, we obtain
\[
   K_{2}+K_{3} \ll
   x^{\eps}(Q^{\ell-\sigma}x+Q^{\ell+k/2-\sigma/2}x^{5/6}
    + Q^{\ell+(k-1)\sigma/2}x^{1-\sigma/6}).
\]

\textbf{Second range:}
$Q^{k+3-\sigma}\leq x\leq Q^{2k+\sigma}$.
There, we choose $U=x^{2/3}Q^{-B}$ with $B=(\sigma+2k)/3$.
Hence
\begin{align*}
  K_{2}+K_{3}&\ll x^{\eps}(Q^{\ell+k}x^{1/2} + 
 Q^{\ell-\sigma}x+Q^{\ell+k\sigma}x^{1-\sigma/2}\\
&+Q^{\ell+(k-\sigma)/2}x^{2/3}Q^{B/2}+Q^{\ell+k/2+k\sigma/2}x^{2/3-\sigma/6}Q^{B(1-\sigma)/2}\\
&+ Q^{\ell+(k-1)\sigma/2}x^{1-\sigma/3}Q^{B\sigma/2}).
\end{align*}
Now the dominating summand in this bound is
$Q^{\ell+(k-\sigma)/2+B/2}x^{2/3}$ within this range,
and it also dominates the bound for $K_{1}$.

This shows the theorem.
\end{proof}

\section{Proof of Theorem~\ref{th:BV}}
\label{sec:thproof}

Before starting with the proof of Theorem~\ref{th:BV},
we deduce an auxiliary result from the previous sections. 

\begin{lemma}
\label{BVLemma}
  For a polynomial $P$ as in Assumptions~\ref{Assu}
of degree $k$ in $\ell$ variables we have
\[
   E:= Q^{-\ell} \sum_{\uq\sim Q} \frac{P(\uq)}{\ph(P(\uq))}
    \sideset{}{^{*}}\sum_{\chi_{1}(P(\uq))}
    \sup_{y\leq x}|\psi(x,\chi_{1})|\ll x^{1-\delta}
\]
for any small value $\delta>0$, assuming that
$x^{\eps/\sigma}\ll Q\leq x^{(1/3-2\eps)/(k-\sigma)}$ 
for any fixed $\eps>\delta$.
\end{lemma}

\begin{proof}
  Theorem~\ref{thm:polyBMVT} yields the bound $E\ll
  Q^{-\ell}x^{\eps}\tilde{\Delta}(Q,x)$ with
  \[
  \tilde{\Delta}(Q,x):=Q^{\ell-\sigma}x
  +Q^{\ell+(k-\sigma)/2}x^{5/6}+Q^{\ell+(k-1)\sigma/2}x^{1-\sigma/6}
  \]
  assuming that $Q\leq x^{1/(2k+\sigma)}$: hence
  \[
      E\ll
      x^{\eps}(Q^{-\sigma}x
      +Q^{(k-\sigma)/2}x^{5/6}
      +Q^{(k-1)\sigma/2}x^{1-\sigma/6})\ll x^{1-\delta}
  \]
  holds in the range $x^{(\eps+\delta)/\sigma}\ll Q\leq
  x^{(1/3-2(\eps+\delta))/(k-\sigma)}$ for $Q$. Now replace
  $\eps+\delta$ by $\eps>\delta$ in the upper and lower bound.
\end{proof}

Now we give the main proof.

\begin{proof}[Proof of Theorem~\ref{th:BV}]
Let $P$ be the polynomial of degree $2k$ as in the statement 
and $\sigma=1/(4rk)$ with $r=\binom{2k+\ell-1}{\ell}-1$,
$\eps'=2\eps/(2k-\sigma)$.

Let $Q,x > 1$, and for a character $\chi$ mod 
$P(\uq)$, we set
$\psi'(x,\chi):=\psi(x,\chi)$ if $\chi$ is different from the
principal character $\chi_{0}$, and
$\psi'(x,\chi_{0})$ $:=\psi(x,\chi_{0})-x$ otherwise.
 So for $y\leq x$ we have
  \[
      E(y;P(\uq),a)=\psi(y;P(\uq),a)-\frac{y}{\ph(P(\uq))}
    = \frac{1}{\ph(P(\uq))} \sum_{\chi(P(\uq))} \overline{\chi}(a) 
    \psi'(y,\chi), 
  \]
and hence
  \[
     \max_{\substack{ a \md P(\uq)\\\gcd(a,P(\uq))=1}} |E(y;P(\uq),a)|
    \leq \frac{1}{\ph(P(\uq))} \sum_{\chi(P(\uq))} 
   |\psi'(y;P(\uq),a)|.
  \]
  If $\chi$ is induced by the primitive character $\chi_{1}$ 
  modulo $d$ with $d\mid P(\uq)$, then 
  $\psi(y,\chi)-\psi'(y,\chi_{1})\ll (\log (yP(\uq)))^{2}$.
Thus
  \[     
    \max_{\substack{ a \md P(\uq)\\\gcd(a,P(\uq))=1}} |E(y;P(\uq),a)|
    \ll \frac{1}{\ph(P(\uq))}
\sum_{\substack{\chi(P(\uq))\\\text{ind. of } \chi_{1}\\ 
   \md d\mid P(\uq)}} 
    |\psi'(y,\chi_{1})| + (\log x)^{2},
  \]
  and so
  \begin{multline*}
    \sum_{\uq\sim Q}  G_{\uq}\frac{\ph(P(\uq))}{Q^{\ell}}
   \sup_{y\leq x} \max_{\substack{ a \md P(\uq)\\\gcd(a,P(\uq))=1}} 
      |E(y;P(\uq),a)|\\
   \ll \sum_{\uq\sim Q} \frac{G_{\uq}}{Q^{\ell}}
   \sum_{\substack{\chi(P(\uq))\\\text{ind. of } \chi_{1}\\ 
   \md d\mid P(\uq)}} \sup_{y\leq x} |\psi'(y,\chi_{1})|+
    Q^{2k} (\log x)^{2+k},
  \end{multline*}
  where the term $Q^{2k}(\log x)^{2+k}$ is clearly admissible in the 
  considered $Q$-range, since there, $Q\ll x^{1/4k}$.

\zeile
  Now due to the assigned weight $G_{\uq}$, each $P(\uq)$ is
  squarefree, so $\uq$ is so that each
  $q_{u(i)}^{2}+q_{v(i)}^{2}=p_{i}$ is a prime, $1\leq i\leq k$,
  and these primes are pairwise different.
  Hence, for a divisor $d$ of $P(\uq)$, we have $d=1$ or 
  $d=\tilde{P}(\uq)$ for a polynomial $\tilde{P}$ that 
  divides $P$ which is of a similar shape as $P$ itself.
  We split the sum over $\chi$ according to these two
  cases.

  In the first case, when $d=1$, we have $\chi_{1}\equiv 1$
(the constant $1$ character) and $\chi=\chi_{0}$ mod $P(\uq)$ is
unique, hence
\[
    \sum_{\uq\sim Q} \frac{G_{\uq}}{Q^{\ell}} \sup_{y\leq x} 
   |\psi(y,\chi_{0})-y|\ll x(\log x)^{-A}
\]
for any $A>0$ by the prime number theorem.

In the second case, we obtain the expression
\begin{multline*}
  \sum_{1<d\leq P(\uq)} 
   \sum_{\substack{\uq\sim Q\\ d\mid P(\uq)}} \frac{G_{\uq}}{Q^{\ell}}
   \sideset{}{^{*}}\sum_{\chi_{1}(d)} \sup_{y\leq x}
    |\psi'(y,\chi_{1})| \\
  \ll  \sum_{\tilde{P}\mid P} 
  \sum_{\uq\sim Q} \frac{G_{\uq}}{Q^{\ell}}
 \sideset{}{^{*}}\sum_{\chi_{1}(\tilde{P}(\uq))} \sup_{y\leq x}
    |\psi'(y,\chi_{1})|.
\end{multline*}

Now for every $\tilde{P}$, we apply Lemma~\ref{BVLemma}
and together with the trivial observation 
$\#\{\uq\sim Q\}\ll Q^{\ell}$, we bound this expression by
\[
  \ll \sum_{\tilde{P}\mid P} x^{1-\delta}(\log x)^{k+1},
\]
which holds for $x^{\eps/\sigma}\ll Q\ll x^{(1/3-2\eps)/(2k-\sigma)}$
and $\eps>\delta>0$.
We used that
the exponents $(1/3-2\eps)/(2k(\tilde{P})-\sigma(\tilde{P}))$
are all $\geq (1/3-2\eps)/(2k-\sigma)$,
since for $\tilde{P}\mid P$
with $\tilde{P}\neq P$,
we have $\deg\tilde{P}\leq \deg P-2$.
This yields the desired estimate of the theorem since there are
only $\ll_{k} 1$ many divisor polynomials of $P$ in $\Z[\ux]$.
\end{proof}

\section{Proof of Theorem~\ref{th:cases}}
\label{sec:last}

The proof depends on the following result of Fouvry and
Iwaniec in \cite{FI}.

\begin{theorem}[Fouvry and Iwaniec]
\label{th:FI}
  Let $(\lambda_{\ell})$ be a complex sequence with 
  $|\lambda_{\ell}|\leq 1$. Then, if $A,x>1$, we have
\[
   \sum_{\ell^{2}+m^{2}\leq x} \lambda_{\ell} \Lambda(\ell^{2}+m^{2})
    = \sum_{\ell^{2}+m^{2}\leq x} \lambda_{\ell}
    \frac{4c}{\pi}\theta(\ell)+O_{A}\Big(\frac{x}{(\log x)^{A}}\Big)
\]
with $\theta(\ell):=\prod_{p\mid \ell}(1-\chi(p)/(p-1))^{-1}$
and $c:=\prod_{p}(1-\frac{\chi(p)}{(p-1)(p-\chi(p))})$.
\end{theorem}
Note that $\theta(\ell)\geq \ph(\ell)/\ell\gg 1/(\log \log\ell)$ holds.

Using this theorem, the proof of Theorem~\ref{th:cases} can be worked out
by induction on $k$. 

\begin{proof}[Proof of Theorem~\ref{th:cases}]
It suffices to prove Theorem~\ref{th:cases} 
without the factor $\mu^{2}(P(\uq))$, that is
\begin{equation}
  \label{eq:sta}
  \sum_{\uq\sim Q} \Lambda(q_{u(1)}^{2}+q_{v(1)}^{2})\cdots
   \Lambda(q_{u(k)}^{2}+q_{v(k)}^{2}) \gg \frac{Q^{\ell}}{(\log Q)^{C}}
\end{equation}
for some constant $C>0$, what
can be seen as follows: In the difference of the left hand sides, 
at least one of the $\Lambda$-arguments must be a prime power, say, 
that this is in the $i$-th $\Lambda$-factor. 
Let $\uq'$ denote the $(\ell-2)$-tuple obtained from
$\uq$ by deleting the coordinates with $u(i)$ and $v(i)$.
Then the deviation can be bounded by
\begin{multline*}
   \sum_{\substack{m\leq  8Q^{2} \\ m=p^{k}\\k\geq 2}} \Lambda(m)
   \sum_{\substack{q_{u(i)},q_{v(i)}\sim Q\\ q_{u(i)}^{2}+q_{v(i)}^{2}=m}}
   \sum_{\uq'\sim Q} \prod_{\substack{j=1\\j\neq i}}^{k}
   \Lambda(q_{u(j)}^{2}+q_{v(j)}^{2})\\  \ll Q\log Q\cdot
   Q^{\ell-2}(\log Q)^{k-1},
\end{multline*}
what is admissible for the desired lower bound of the theorem.

Now we give the proof of \eqref{eq:sta} by induction.
Let $k=1$: then we have to show that
 \[
     \sum_{q_{1},q_{2}\sim Q} 
   \Lambda(q_{1}^{2}+q_{2}^{2})
   \gg \frac{Q^{2}}{(\log Q)^{C}}
\]
holds for a constant $C>0$.
A simple geometric argument shows that
\begin{multline*}
    \sum_{q_{1},q_{2}\sim Q}\Lambda(q_{1}^{2}+q_{2}^{2})
   =\sum_{2Q^{2}<q_{1}^{2}+q_{2}^{2}\leq 8Q^{2}} \Lambda(q_{1}^{2}+q_{2}^{2})
    -2\sum_{\substack{2Q^{2}<q_{1}^{2}+q_{2}^{2}\leq
        8Q^{2}\\q_{1}>2Q}} \Lambda(q_{1}^{2}+q_{2}^{2})\\
  -2\Big(\sum_{\substack{2Q^{2}<q_{1}^{2}+q_{2}^{2}\leq
      5Q^{2}\\q_{1}\leq Q}} \Lambda(q_{1}^{2}+q_{2}^{2})
  - \sum_{\substack{2Q^{2}<q_{1}^{2}+q_{2}^{2}\leq 5Q^{2}\\q_{1}\geq
      2Q}} \Lambda(q_{1}^{2}+q_{2}^{2})\Big).
\end{multline*}

Each sum is of the form $\sum_{y<q_{1}^{2}+q_{2}^{2}\leq x}
\lambda_{q_{1}}\Lambda(q_{1}^{2}+q_{2}^{2})$ where
$\lambda_{q_{1}}$ is the characteristic function of the 
conditions on $q_{1}$. Hence, to each sum, Theorem~\ref{th:FI} applies
giving
\begin{align*}
  \sum_{y<q_{1}^{2}+q_{2}^{2}\leq x} \lambda_{q_{1}}
  \frac{4c}{\pi} \theta(q_{1}) 
 +O_{A}\Big(\frac{Q^{2}}{(\log Q)^{A}}\Big),
\end{align*}
and the main terms combine again. In such a way, we obtain
\begin{align*}
     \sum_{q_{1},q_{2}\sim Q} \Lambda(q_{1}^{2}+q_{2}^{2}) 
   &= \sum_{q_{1},q_{2}\sim Q} \frac{4c}{\pi} \theta(q_{1}) +
   O_{A}\Big(\frac{Q^{2}}{(\log Q)^{A}}\Big) \\
  &\gg \sum_{q_{1},q_{2}\sim Q} \frac{1}{\log\log q_{1}}+
 O_{A}\Big(\frac{Q^{2}}{(\log Q)^{A}}\Big)
 \gg \frac{Q^{2}}{\log Q},
\end{align*}
fixing 
$A>1$ in the last step.

Now let $k\geq 2$ and assume that estimate~\eqref{eq:sta} holds 
for $k-1$: we proceed to
show it for $k$. Surely, $\ell\geq 2$, 
and we divide the left hand side in~\eqref{eq:sta}  
by $(\log Q)^{k-1}Q^{\ell-2}$ and obtain
\[
  \sum_{\uq\sim Q} (\log Q)^{1-k} Q^{2-\ell}\Big(\prod_{i=2}^{k}
   \Lambda(q_{u(i)}^{2}+q_{v(i)}^{2})\Big)\Lambda(q_{u(1)}^{2}+q_{v(1)}^{2}).
\]
Let $\uq'$ be obtained from $\uq$ by deleting the 
coordinates with index $u(1)$ and $v(1)$.
By assumption, one of the indices $u(1)$ and $v(1)$ does not occur in
$\{u(2),\dots,u(k),v(2),\dots,v(k)\}$,
assume w.l.o.g.\ that this is $u(1)$.
Then the considered sum transforms into
\[
  \sum_{q_{u(1)},q_{v(1)}\sim Q}\lambda_{q_{v(1)}}\Lambda(q_{u(1)}^{2}+q_{v(1)}^{2})
\]
with
\[
   \lambda_{q_{v(1)}}:=(\log (8Q^{2}))^{1-k}Q^{2-\ell} \sum_{\uq'\sim Q}
 \prod_{i=2}^{k}\Lambda(q_{u(i)}^{2}+q_{v(i)}^{2}) \leq 1,
\]
which does not depend on $u(1)$. By induction hypothesis, we have
\begin{equation}
\label{eq:lamb}
  \sum_{q_{v(1)}\sim Q} \lambda_{q_{v(1)}} \gg (\log Q)^{1-k} Q^{2-\ell}
  \frac{Q^{\ell-1}}{(\log Q)^{C}}=Q(\log Q)^{1-k-C}
\end{equation}
for some constant $C>0$.

We apply Theorem~\ref{th:FI} 
similarly
to the case $k=1$, which yields
\begin{align*}
  &\sum_{q_{u(1)},q_{v(1)}\sim Q} \lambda_{q_{v(1)}}
  \Lambda(q_{u(1)}^{2}+q_{v(1)}^{2})\\
&\geq\sum_{\substack{13Q^{2}/4<q_{u(1)}^{2}+q_{v(1)}^{2}\leq 5Q^{2}\\
    Q<q_{v(1)}\leq 3Q/2}} \lambda_{q_{v(1)}}
    \Lambda(q_{u(1)}^{2}+q_{v(1)}^{2}) \\
   &= \sum_{\substack{13Q^{2}/4<q_{u(1)}^{2}+q_{v(1)}^{2}\leq 5Q^{2}\\
       Q<q_{v(1)}\leq 3Q/2}} \lambda_{q_{v(1)}} \frac{4c}{\pi} 
      \theta(q_{v(1)}) + O_{A}\Big(\frac{Q^{2}}{(\log Q)^{A}}\Big) \\
  &\gg \sum_{\substack{13Q^{2}/4<q_{u(1)}^{2}+q_{v(1)}^{2}\leq
      5Q^{2}\\  Q<q_{v(1)}\leq 3Q/2}}
  \frac{\lambda_{q_{v(1)}}}{\log\log q_{v(1)}}+
  O_{A}\Big(\frac{Q^{2}}{(\log Q)^{A}}\Big)
  \gg \frac{Q^{2}}{(\log Q)^{C'}}
\end{align*}
for a constant $C'>0$ and a $A>1$ chosen large enough,
where we used \eqref{eq:lamb} in the last estimate 
(adjusted to an appropriate scaled box that is contained 
in the considered region). 
This concludes the proof of \eqref{eq:sta} and therefore
of Theorem~\ref{th:cases}.
\end{proof}




\begin{thebibliography}{20}
\bibitem{Bak} R.~C.~Baker,
Primes in arithmetic progressions to spaced moduli,
Acta Arith.\ 153 (2012), no.\ 2, 133--159. 
\bibitem{Bo} E.~Bombieri, On the large sieve, Mathematika 12,
  201--225.
\bibitem{Brue} J.~Br\"udern, Einf\"uhrung in die analytische
  Zahlentheorie, Springer Lehrbuch, 1995.
\bibitem{Ell} P.~D.~T.~A.~Elliott, Primes in short arithmetic 
progressions with rapidly increasing differences, 
Trans.\ Amer.\ Math.\ Soc.\ 353 (2001), no.\ 7, 2705--2724.
\bibitem{EHconj}
P.~D.~T.~A.~Elliott and H.~Halberstam, A conjecture in prime number
theory. In Symposia Mathematica, Vol.\ IV 
(INDAM, Rome, 1968/69), pages 59­-72. Academic Press, London, 1970.
\bibitem{FI} Fouvry and Iwaniec, Gaussian primes,
Acta Arith.\ 79 (1997), no.\ 3, 249--287.
\bibitem{KH} K.~Halupczok, 
Large sieve inequalities with general polynomial moduli, 
Q. J. Math. 66 (2015), no. 2, 529--545; doi: 10.1093/qmath/hav011 
\bibitem{MikPen} H.~Mikawa and T.~P.~Peneva,
Primes in arithmetic progressions to spaced moduli,
Arch.\ Math.\ (Basel) 84 (2005), no.\ 3, 239--248. 
\bibitem{PPW} 
S.~T.~Parsell, S.~M.~Prendiville and T.~D.~Wooley,
Near-optimal mean value estimates for multidimensional Weyl sums,
Geom.\ Funct.\ Anal.\ 23 (2013), no.~6, 1962--2024. 
\bibitem{VauT} R.~C.~Vaughan,
The Bombieri--Vinogradov Theorem, AIM discussion paper, November
2005. Available at \url{http://www.personal.psu.edu/rcv4/Bombieri.pdf}
\bibitem{AIV1}  A.~I.~Vinogradov,
On the density hypothesis for Dirichlet
L--series, Izv.\ Akad.\ Nauk SSSR, Ser.\ Mat.\ 29, 903--934 (1965).
\bibitem{AIV2} A.~I.~Vinogradov,
Corrections to the work of A.~I.~Vinogradov `On the density hypothesis
for Dirichlet L--series',
Izv.\ Akad.\ Nauk SSSR, Ser.\ Mat.\ 30, 719--729.
\bibitem{Zha} Y.~Zhang,
Bounded gaps between primes,
Ann.\ of Math.\ (2) 179 (2014), no.\ 3, 1121--1174. 
\end{thebibliography}
\end{document}